\documentclass{amsart}
\usepackage{amsthm, amssymb, amsfonts, amscd}
\usepackage{graphicx}
\usepackage{stmaryrd} 
\usepackage{bbold}
\usepackage{verbatim}
\usepackage[all,arc]{xy}
\usepackage{enumerate}
\usepackage{mathrsfs, mathabx}
\usepackage{xcolor}
\usepackage{hyperref}
\usepackage{soul}
\usepackage{comment}
\usepackage{esint} 
\usepackage[all]{xy}
\usepackage[margin=1.0in]{geometry}
\usepackage[
   open,
   openlevel=2,
   atend
 ]{bookmark}[2011/12/02]
\usepackage{graphics}
\bookmarksetup{color=blue}
\allowdisplaybreaks

\usepackage{enumitem}
\usepackage{tikz-cd}
\usepackage[T1]{fontenc}
\usepackage{mathtools}
\usepackage{setspace}
\allowdisplaybreaks
\DeclareMathOperator{\im}{im}
\DeclareMathOperator{\cok}{cok}
\newcommand{\mr}{\mathrm}
\newcommand{\ind}{{\mr{ind}}}
\newcommand{\Ind}{{\mr{Ind}}}
\newcommand{\Cl}{{\mr{Cl}}}
\newcommand{\Gal}{\mr{Gal}}
\newcommand{\Hom}{\mr{Hom}}
\newcommand{\Aut}{\mr{Aut}}
\newcommand{\GL}{\mr{GL}}

\newcommand{\rank}{\mr{rank}}
\newcommand{\res}{\mr{R}}
\newcommand{\nm}{\mr{N}}
\newcommand{\NF}{\mr{NF}}
\newcommand{\tor}{\mr{tor}}

\newcommand{\cond}{\mr{Cond}}
\newcommand{\lcm}{\mr{lcm}}
\newcommand{\ord}{\mr{ord}}
\newcommand{\Sqf}{\mr{Sqf}}

\newcommand{\bb}{\mathbb}
\newcommand{\bZ}{\bb{Z}}
\newcommand{\bQ}{\bb{Q}}
\newcommand{\bC}{\bb{C}}
\newcommand{\R}{\bb{C}}
\newcommand{\G}{\bb{G}}
\newcommand{\CC}{\mathcal{C}}
\newcommand{\ve}{\varepsilon}

\theoremstyle{definition} 
\newtheorem{theorem}{Theorem}[section]
\newtheorem{corollary}[theorem]{Corollary}
\newtheorem{proposition}[theorem]{Proposition}
\newtheorem{lemma}[theorem]{Lemma}
\newtheorem{conjecture}[theorem]{Conjecture}
\newtheorem{remark}[theorem]{Remark}
\newtheorem{example}[theorem]{Example}

\begin{document}
	
\title[Counting algebraic tori over $\bQ$ by Artin conductor]{Counting algebraic tori over $\bQ$ by Artin conductor}
\author{Jungin Lee}
\date{}
\address{J. Lee -- Department of Mathematics, Ajou University, Suwon 16499, Republic of Korea}
\email{jileemath@ajou.ac.kr}
	
\begin{abstract}
In this paper we count the number $N_n^{\tor}(X)$ of $n$-dimensional algebraic tori over $\bQ$ whose Artin conductor of the associated character is bounded by $X$. This can be understood as a generalization of counting number fields of a given degree by discriminant. 
We suggest a conjecture on the asymptotics of $N_n^{\tor}(X)$ and prove that this conjecture follows from Malle's conjecture for tori over $\bQ$. We also prove that $N_2^{\tor}(X) \ll_{\ve} X^{1 + \ve}$, and this upper bound can be improved to $N_2^{\tor}(X) \ll_{\ve} X (\log X)^{1 + \ve}$ under the assumption of the Cohen--Lenstra heuristics for $p=3$.
\end{abstract}
\maketitle

\vspace{-7mm}
\section{Introduction} \label{Sec1}

\subsection{Counting number fields by discriminant} \label{Sub11}

Counting number fields by discriminant is one of the most important topics in arithmetic statistics. Let $L/K$ be an extension of number fields. (Throughout the paper, all number fields are taken to be subfields of a fixed algebraic closure of $\bQ$.) Denote the absolute value of the discriminant of $L$ by $D_L$. Also take $\delta_{L/K}$ to be the relative discriminant ideal of $L/K$, $\nm_{K/\bQ}$ to be the absolute norm and denote $\displaystyle D_{L/K} := \nm_{K/\bQ}(\delta_{L/K})$. Let $C_m$ be the cyclic group of order $m$, $D_m$ be the dihedral group of order $2m$ and $S_m$ be the symmetric group of degree $m$.

For a number field $K$ and an integer $n \geq 2$, denote by $N_{K, n}(X)$ the number of degree $n$ extensions $L$ of $K$ such that $D_{L/K} \leq X$ and denote $N_{n}(X) := N_{\bQ, n}(X)$ for simplicity. 
A folk conjecture (sometimes attributed to Linnik) states that $N_{K, n}(X) \sim c_{K, n}X$ as $X \rightarrow \infty$ for some constant $c_{K, n}>0$. This conjecture has been proved for $n \leq 5$ by the work of Davenport--Heilbronn \cite{DH71}, Datskovsky--Wright \cite{DW88}, Bhargava \cite{Bha05, Bha10} and Bhargava--Shankar--Wang \cite{BSW15}.

Now let $G \neq 1$ be a transitive subgroup of the symmetric group $S_n$. Denote by $N_{K, n}(X; G)$ the number of degree $n$ extensions $L$ of $K$ such that $D_{L/K} \leq X$ and $\Gal(L^c/K) \cong G$. Here $L^c$ denotes the Galois closure of $L/K$ and the Galois group $\Gal(L^c/K)$ is permutation-isomorphic to $G$. Denote $N_{n}(X; G) := N_{\bQ, n}(X; G)$ for simplicity. Malle's conjecture \cite{Mal02, Mal04} states that 
\begin{equation} \label{eq1a}
    N_{K, n}(X; G) \sim c_{K, G} X^{\frac{1}{a(G)}} (\log X)^{b(K, G)-1}
\end{equation}
for some positive integers $a(G), b(K,G)$ and a constant $c_{K, G}>0$. There is also Malle's weak conjecture which states that 
$X^{\frac{1}{a(G)}} \ll N_{K, n}(X; G) \ll X^{\frac{1}{a(G)}+\ve}$ for any $\ve>0$.

The numbers $a(G)$ and $b(K, G)$ are defined as follows. For any $g \in G \leq S_n$, define the index of $g$ by 
$$
\ind(g) := n - \text{the number of orbits of } g \text{ on } \left \{ 1, 2, \cdots, n \right \}
$$
and let $a(G) := \min_{g \in G \setminus \left \{ 1 \right \}} \ind(g)$. For any field $k$, denote its algebraic closure by $\overline{k}$ and its absolute Galois group by $G_k$.
Let $\chi : G_{\bQ} \to (\bZ / |G|\bZ)^{\times}$ be the mod-$|G|$ cyclotomic character, and define an action of $G_{\bQ}$ on $G$ by
$$
\sigma(g) := g^{\chi(\sigma)}
$$
for every $g \in G$ and $\sigma \in G_{\bQ}$. This induces an action of $G_{\bQ}$ on the set $\Cl(G)$ of conjugacy classes of $G$, and the action restricts to a subgroup $G_K$ of $G_{\bQ}$.

Define the $K$-conjugacy classes of $G$ to be the orbits of the action of $G_K$ on $\Cl(G)$.
Since all elements of a $K$-conjugacy class $C$ have the same index (cf. \cite[p. 130]{Mal04}), its index can be defined by the index of any element of $C$. 
The number $b(K, G)$ is defined to be the number of $K$-conjugacy classes of $G$ whose index is $a(G)$. Since $a(G)=1$ implies that $b(K, G)=1$ \cite[Lemma 2.2]{Mal04}, Malle's conjecture implies the asymptotics $N_{K, n}(X) \sim c_{K, n}X$.

Malle's conjecture for abelian extensions follows from the work of Mäki \cite{Mak85} for $K=\bQ$ and Wright \cite{Wri89} for general $K$. The cases $G=S_n$ for $3 \leq n \leq 5$ follow from the work of Davenport--Heilbronn \cite{DH71}, Datskovsky--Wright \cite{DW88}, Bhargava \cite{Bha05, Bha10} and Bhargava--Shankar--Wang \cite{BSW15}. The product of these two cases, i.e. $G=S_d \times A$ and $n=d \left | A \right |$ for $3 \leq d \leq 5$ and an abelian group $A$, was recently proved by Wang \cite{Wan21} and Masri--Thorne--Tsai--Wang \cite{MTTW20}. See also \cite{AOWW26, BW08, CDO02, FK21, Klu12, KP21, SV26} for more cases where the conjecture has been settled.

In the opposite direction, Klüners \cite{Klu05} found a counterexample $G = C_3 \wr C_2 \cong S_3 \times C_3 \leq S_6$ for Malle's conjecture. Türkelli \cite{Tur15} proposed a modified version of Malle's conjecture (with the same $a(G)$ and a different $b(K, G)$) which takes into account Klüners' counterexample and agrees with the heuristics in the function field case. Alberts \cite{Alb21} provided more evidences that Türkelli's modification is correct.

\subsection{Counting algebraic tori over \texorpdfstring{$\bQ$}{Q} by Artin conductor} \label{Sub12}

In this subsection, we explain how counting number fields can be regarded as a special case of counting (algebraic) tori over $\bQ$. The functor from the category of tori over $\bQ$ to the category of $G_{\bQ}$-lattices which maps a torus $T$ to its character group $X^*(T) := \Hom_{\overline{\bQ}} (T_{\overline{\bQ}}, \G_{m, \overline{\bQ}})$ is an anti-equivalence of categories. Let $T$ be an $n$-dimensional torus over $\bQ$ whose splitting field is $L$. The $G_{\bQ}$-action on $X^*(T)$ gives the representation
$$
\rho_T : G_{\bQ} \rightarrow \Aut (X^*(T)) \cong \GL_n(\bZ)
$$
such that $\ker (\rho_T) = G_L$ and $G_T := \im (\rho_T)$ is isomorphic to $\Gal(L/\bQ)$. Note that $\rho_T$ and $G_T$ are well-defined only up to conjugation since the isomorphism $\Aut (X^*(T)) \cong \GL_n(\bZ)$ depends on the choice of basis. Now let $C(T)$ denote the Artin conductor of the character associated to the representation 
$$
\rho : \Gal(L/\bQ) \cong G_{\bQ}/G_L \rightarrow \Aut (X^*(T)_{\bQ}) \cong \GL_n(\bQ)
$$
induced by $\rho_T$. It is always a positive integer.

Let $K$ be a number field of degree $n$, $K^c$ be the Galois closure of $K/\bQ$, $G=\Gal (K^c / \bQ)$ and $G'=\Gal(K^c / K) \leq G$.  
Consider an $n$-dimensional torus $T = \res_{K/\bQ} \G_m$ (Weil restriction of $\G_m$) over $\bQ$. The splitting field of $T$ is $K^c$ and its character group is given by 
$$
X^*(T) \cong \Ind_{G_K}^{G_{\bQ}} X^*(\G_m) = \Ind_{G_K}^{G_{\bQ}} \bZ
$$
(cf. \cite[Proposition 11.4.22]{Spr98} and \cite[12.4.5]{Spr98}).

Since $T$ splits over $K^c$, the $G_{\bQ}$-action on $X^*(T)$ factors through $G$ and $X^*(T) \cong \Ind_{G'}^G \bZ$ as $G$-modules. 
Therefore the character of the representation $\rho : G \rightarrow \Aut (X^*(T)_{\bQ}) \cong \GL_n(\bQ)$ is $\Ind_{G'}^G \mathbf{1}_{G'}$ (induced character of the trivial character $\mathbf{1}_{G'}$), whose Artin conductor is $C(\res_{K/\bQ} \G_m)=D_K$ by \cite[Corollary VII.11.8]{Neu99}. This shows that counting tori over $\bQ$ of dimension $n$ by Artin conductor can be understood as a generalization of counting number fields of degree $n$ by discriminant.

Denote by $N_n^{\tor}(X)$ the number of the isomorphism classes of tori over $\bQ$ of dimension $n$ such that $C(T) \leq X$. For a finite subgroup $H \neq 1$ of $\GL_n(\bZ)$, denote by $N_n^{\tor}(X; H)$ the number of such tori $T$ over $\bQ$ such that $G_T$ is conjugate to $H$ in $\GL_n(\bZ)$.

\subsection{Main results and the structure of the paper} \label{Sub13}

We start with some preliminaries on counting algebraic tori. First we explain how to classify tori $T$ over $\bQ$ such that $G_T$ is conjugate to given finite subgroup $H \neq 1$ of $\GL_n(\bZ)$ in Section \ref{Sub21}. We discuss the computation of the Artin conductor $C(T)$ in Section \ref{Sub22}. After that, we review some backgrounds on counting number fields in Section \ref{Sub23} and \ref{Sub25}.

In Section \ref{Sec3} we provide conjectures on the number of the isomorphism classes of tori over $\bQ$. First we suggest an asymptotics of the number $N_n^{\tor}(X)$.

\begin{conjecture} \label{conj1a}
(Conjecture \ref{conj3a}) For every $n \geq 1$, there exists a constant $c_n>0$ satisfying
\begin{equation}
N_n^{\tor}(X) \sim c_n X (\log X)^{n-1}.
\end{equation}
\end{conjecture}

Next we provide an analogue of Malle's conjecture for tori over $\bQ$. This conjecture is not new; it is a direct consequence of a more general conjecture of Ellenberg and Venkatesh \cite[Question 4.3]{EV05}. There is a further generalization of the conjecture by Ellenberg, Satriano and Zuerick-Brown \cite[Conjecture 4.14]{ESZB21} which specializes both to Batyrev--Manin conjecture and to Malle’s conjecture.

\begin{conjecture} \label{conj1b}
(Conjecture \ref{conj3b}) For every $n \geq 1$ and a finite subgroup $1 \neq H \leq \GL_n(\bZ)$, 
\begin{equation}
N_n^{\tor}(X; H) \sim c_H X^{\frac{1}{a(H)}} (\log X)^{b(H)-1}
\end{equation}
where $a(H)$ and $b(H)$ are the positive integers defined in Section \ref{Sec3}, and $c_H>0$ is a constant depending only on $H$.
\end{conjecture}

We present three remarks on this conjecture. First, the above conjecture is a generalization of Malle's conjecture (Remark \ref{rmk3c}). Secondly, the conjecture is compatible with the direct product of $H$ (Remark \ref{rmk3d}). Finally and most importantly, if Conjecture \ref{conj1b} holds for every finite nontrivial subgroup of $\GL_n(\bZ)$, then Conjecture \ref{conj1a} is true for $n$ (Corollary \ref{cor3f}).

In Section \ref{Sec4} we concentrate on the $2$-dimensional case. Based on the results in Section \ref{Sub21} and \ref{Sub22}, we classify $2$-dimensional tori over $\bQ$ and compute their Artin conductors in Section \ref{Sub41}. Section \ref{Sub42} to \ref{Sub44} is devoted to the asymptotics of the number of $2$-dimensional tori over $\bQ$. Since Conjecture \ref{conj1b} can be easily proved when $n=2$ and $H \neq H_{12, A}$, the essential new result is the estimation of $N_2^{\tor}(X; H_{12, A})$. Here $H_{12, A}$ is a finite subgroup of $\GL_2(\bZ)$ isomorphic to the dihedral group $D_6$, which appears in Section \ref{Sub41}.

The following theorem summarizes the main results of Section \ref{Sec4}. The first part of the theorem is unconditional, and the second part is under the assumption of the Cohen--Lenstra heuristics for $p=3$. See Conjecture \ref{conj43c} for a precise version of the Cohen--Lenstra heuristics used in this paper. Throughout the paper, we write $f(X) \ll_{\ve} g(X,\ve)$ to mean that for every $\ve > 0$, there exists a constant $C_{\ve}>0$ (depending on $\ve$) such that $f(X) \le C_{\ve}g(X,\ve)$.

\begin{theorem} \label{thm1c}
\begin{enumerate}
    \item (Theorem \ref{thm44a}) 
    \begin{subequations} 
\begin{align}
X \ll N_2^{\tor}(X; H_{12, A}) & \ll_{\ve} X^{1+\frac{\log 2 + \ve}{\log \log X}} \\
X \log X \ll N_2^{\tor}(X) & \ll_{\ve} X^{1+\frac{\log 2 + \ve}{\log \log X}}. 
\end{align}
\end{subequations}
    
    \item (Theorem \ref{thm44b}) Assume that Cohen--Lenstra heuristics (Conjecture \ref{conj43c}) holds for $p=3$ and every $\alpha >0$. Then we have
\begin{equation} 
N_2^{\tor}(X; H_{12, A}) \leq N_2^{\tor}(X) 
\ll_{\ve} X (\log X)^{1 + \ve}.
\end{equation}
\end{enumerate}
\end{theorem}

Here we briefly explain why the estimation of the asymptotics of $N_2^{\tor}(X; H_{12, A})$ is difficult compared to the case $N_2^{\tor}(X; H_{8, A})$. 
The asymptotics of $N_2^{\tor}(X; H_{8, A})$ follows from the work of Altuğ, Shankar, Varma and Wilson \cite{ASVW21}, where both analytic techniques and geometry-of-numbers methods were used. In particular, the parametrization of $D_4$-quartic fields via certain pairs of ternary quadratic forms following Bhargava \cite{Bha05} and Wood \cite{Woo09} was essential. However, such a parametrization is not yet known for $D_6$-sextic fields. In the sequel of the paper, we estimate $N_2^{\tor}(X; H_{12, A})$ by understanding a $D_6$-sextic field as a compositum of an $S_3$-cubic field and a quadratic field.

\section{Preliminaries} \label{Sec2}

\subsection{Classification of tori over \texorpdfstring{$\bQ$}{Q}} \label{Sub21}

For an $n$-dimensional torus $T$ over $\bQ$, we have defined a finite subgroup $G_T \leq \GL_n(\bZ)$ (well-defined up to conjugation) which is isomorphic to the Galois group of the splitting field of $T$. 
In this section, we explain how to compute $n$-dimensional tori $T$ over $\bQ$ such that $G_T$ is conjugate to the given finite subgroup $H$ of $\GL_n(\bZ)$. We explain the general method and provide an explicit computation for one example. A complete classification of $2$-dimensional tori over $\bQ$ will be provided in Section \ref{Sec4}. 

Recall that the functor $T \mapsto X^*(T)$ from the category of tori over $\bQ$ to the category of $G_{\bQ}$-lattices is an anti-equivalence of categories. Its inverse functor maps a $G_{\bQ}$-lattice $M$ to a torus $T_M$ defined by $T_M(R) := \Hom_{G_{\bQ}}(M, R_{\overline{\bQ}}^{\times})$
for any $\bQ$-algebra $R$. 
Suppose that an $n$-dimensional torus $T$ over $\bQ$ corresponds to a $G_{\bQ}$-lattice $M$ of rank $n$. If the splitting field of $T$ is $L$, then we have $\ker(G_{\bQ} \rightarrow \Aut(M)) = G_L$ and
$$
T(R) \cong \Hom_{\Gal(L/\bQ)}(M, R_{L}^{\times}).
$$

Fix a $\bZ$-basis $x_1, \cdots, x_n$ of $M$. An element $\phi \in \Hom_{\Gal(L/\bQ)}(M, R_{L}^{\times})$ is determined by $n$ values 
$$
v_1 := \phi(x_1), \cdots, v_n := \phi(x_n) \in R_{L}^{\times}=(\res_{L/\bQ} \G_m)(R).
$$
For each $g \in \Gal(L/\bQ)$ and $1 \le j \le n$, we have
\begin{equation} \label{eq21a}
g \cdot v_j 
= \phi(g \cdot x_j)
=\phi(\sum_{i=1}^{n} A_{ij}x_i)
= \prod_{i=1}^{n} v_i^{A_{ij}}
\end{equation}
where $A \in G_T$ is the element corresponding to $g \in \Gal(L/\bQ)$ via the isomorphism
$$
\Gal(L/\bQ) \cong G_{\bQ}/G_L = G_{\bQ}/\ker(\rho_T) \cong \im(\rho_T) = G_T.
$$
Note that it is enough to check the relations \eqref{eq21a} for generators of $\Gal(L/\bQ)$. Now we provide an example to illustrate this.

\begin{example} \label{ex21a}
Let $H$ be the finite subgroup of $\GL_2(\bZ)$ generated by two elements $g = \bigl(\begin{smallmatrix}
1 & -1\\ 
1 & 0
\end{smallmatrix}\bigr)$ and $h = \bigl(\begin{smallmatrix}
0 & 1\\ 
1 & 0
\end{smallmatrix}\bigr)$. It is same as the group $H_{12, A}$ in Section \ref{Sec4}. By the relations $g^6=h^2=(gh)^2=1$, $H$ is isomorphic to the dihedral group $D_6$ of order $12$. Let $T$ be a $2$-dimensional torus over $\bQ$ such that $G_T$ is conjugate to $H$ in $\GL_2(\bZ)$ and $L$ be its splitting field. We identify $\Gal(L/\bQ)$ with $H$ so denote the elements of $\Gal(L/\bQ)$ by $g$, $h$, etc. For an extension $E/K$ of number fields, define a torus $T_{E/K}$ over $\bQ$ as in Section \ref{Sub22}.

For a subgroup $A$ of $\Gal(L/\bQ)$, denote the fixed field of $A$ by $L^A$. Let $L_6 = L^{\left \langle g^4h \right \rangle}$, $L_3=L^{\left \langle g^3, gh \right \rangle}$ and $L_2=L^{\left \langle g^2, h \right \rangle}$. By the equation \eqref{eq21a}, $T$ is determined by $v_1, v_2 \in \res_{L/\bQ} \G_m$ such that
$$
g v_1= v_1^{g_{11}}v_2^{g_{21}} = v_1v_2, \;
g v_2= v_1^{g_{12}}v_2^{g_{22}} = v_1^{-1}
$$
and
$$
h v_1 = v_1^{h_{11}}v_2^{h_{21}} = v_2, \;
h v_2= v_1^{h_{12}}v_2^{h_{22}} = v_1.
$$
A simple computation shows that
\begin{align*} 
T = & \left \{ (v_1, v_2) \in (\res_{L/\bQ} \G_m)^2 : gv_1 = v_1v_2, \, gv_2=v_1^{-1}, \, hv_1=v_2 \text{ and } hv_2=v_1 \right \} \\
\cong & \left \{ v_1 \in \res_{L/\bQ} \G_m : gv_1 = v_1 \cdot hv_1 \text{ and } gh v_1 \cdot v_1 = 1 \right \} \\
= & \left \{ v_1 \in \res_{L/\bQ} \G_m : g^2v_1 \cdot v_1 = gv_1  \text{ and } gh v_1 \cdot v_1 = 1 \right \}  \\
= & \left \{ v_1 \in \res_{L/\bQ} \G_m : v_1 \cdot g^3v_1 = 1 \text{ and } v_1 \cdot g^2v_1 \cdot g^4v_1 = 1 \text{ and } gh v_1 \cdot v_1 = 1  \right \}  \\
= & \left \{ v_1 \in \res_{L/\bQ} \G_m : v_1 \cdot g^3v_1 = 1 \text{ and } v_1 \cdot g^2v_1 \cdot g^4v_1 = 1 \text{ and } g^4h v_1 =v_1  \right \}  \\
= & \left \{ v_1 \in \res_{L_6/\bQ} \G_m : v_1 \cdot g^3v_1 = 1 \text{ and } v_1 \cdot g^2v_1 \cdot g^4v_1 = 1 \right \}  \\
= & \left \{ v_1 \in \res_{L_6/\bQ} \G_m : \nm_{L_6/L_3}(v_1)=1 \text{ and } \nm_{L_6/L_2}(v_1)=1 \right \} \\
= & \,\, T_{L_6/L_3} \cap T_{L_6/L_2}.
\end{align*}
We provide more details on the third equality, as it is not trivial. Assume that $gv_1 = v_1 \cdot g^2v_1$. Then $g^2v_1 \cdot gv_1 = (gv_1 \cdot g^3v_1) \cdot (v_1 \cdot g^2v_1)$ so $v_1 \cdot g^3v_1=1$, and thus $v_1 \cdot g^2v_1 \cdot g^4v_1 = gv_1 \cdot g^4v_1 =g(v_1 \cdot g^3v_1) = 1$. Conversely, assume that $v_1 \cdot g^3v_1=1$ and $v_1 \cdot g^2v_1 \cdot g^4v_1=1$. Then $g^3v_1 = g^2v_1 \cdot g^4v_1$ so $gv_1 = v_1 \cdot g^2v_1$. This proves that $gv_1 = v_1 \cdot g^2v_1$ if and only if $v_1 \cdot g^3v_1=1$ and $v_1 \cdot g^2v_1 \cdot g^4v_1=1$, thereby showing that the fourth equality holds.
\end{example}

\subsection{Computation of \texorpdfstring{$C(T)$}{C(T)}} \label{Sub22}

After the classification of tori over $\bQ$, we need to compute the Artin conductor $C(T)$ for each torus $T$ over $\bQ$. This can be done by combining the following two basic facts. 
  
\begin{proposition} \label{prop22a}
\begin{enumerate}
    \item For a number field $K$, $C(\res_{K/\bQ} \G_m) = D_K$. 
    
    \item If $1 \rightarrow T_1 \rightarrow T_2 \rightarrow T_3 \rightarrow 1$ is an exact sequence of tori over $\bQ$, then $C(T_2)=C(T_1)C(T_3)$. 
\end{enumerate}
\end{proposition}

\begin{proof}
\begin{enumerate}
    \item See Section \ref{Sub12}. 

    \item The exact sequence gives an isogeny $T_2 \sim T_1 \times T_3$ so 
    $$
    X^*(T_2)_{\bQ} \cong X^*(T_1 \times T_3)_{\bQ} \cong X^*(T_1)_{\bQ} \times X^*(T_3)_{\bQ},
    $$
    which implies that $C(T_2)=C(T_1)C(T_3)$. \qedhere
\end{enumerate}
\end{proof}

We introduce some notations for tori over $\bQ$. For an extension $L/K$ of number fields, 
$$
\res_{L/K}^{(1)}\G_m := \ker(\res_{L/K} \G_m \xrightarrow{\nm_{L/K}} \G_m)
$$
is a torus over $K$ of dimension $[L:K]-1$ and
$$
T_{L/K} := \res_{K/\bQ}(\res_{L/K}^{(1)}\G_m).
$$
is a torus over $\bQ$ of dimension $[L:\bQ]-[K:\bQ]$. Note that $T_{L/\bQ} = \res_{L/\bQ}^{(1)}\G_m$ is the norm-one torus. Taking $\res_{K/\bQ}$ on the exact sequence 
$$
1 \rightarrow \res_{L/K}^{(1)}\G_m \rightarrow \res_{L/K} \G_m
\xrightarrow{\nm_{L/K}} \G_m \rightarrow 1
$$
of tori over $K$, we obtain an exact sequence
$$
1 \rightarrow T_{L/K} \rightarrow \res_{L/\bQ}\G_m 
\xrightarrow{\res_{K/\bQ}(\nm_{L/K})} \res_{K/\bQ}\G_m \rightarrow 1
$$
of tori over $\bQ$ (cf. \cite[Section 3.12]{Vos98}). Proposition \ref{prop22a} implies that
\begin{equation} \label{eq22a}
    C(T_{L/K}) = \frac{D_L}{D_K}.
\end{equation}

Now let $K_1$ and $K_2$ be linearly disjoint number fields and $L=K_1K_2$. Recall that $K_1$ and $K_2$ are linearly disjoint if and only if $[L:\bQ]=[K_1:\bQ][K_2:\bQ]$. The following lemma enables us to compute the number $C(T)$ for a torus $T = T_{L/K_1} \cap T_{L/K_2}$.

\begin{lemma} \label{lem22b}
Let $K_1$, $K_2$ and $L$ as above. Then $T_{L/K_1} \cap T_{L/K_2}$ is a torus over $\bQ$ and the sequence
\begin{equation} \label{eq22b}
1 \rightarrow T_{L/K_1} \cap T_{L/K_2}
\rightarrow T_{L/K_1}
\xrightarrow{\res_{K_2/\bQ}(\nm_{L/K_2})} T_{K_2/\bQ}
\rightarrow 1
\end{equation}
is exact. 
\end{lemma}

\begin{proof}
Denote $\alpha := \res_{K_2/\bQ}(\nm_{L/K_2})$. 
We have 
$$
\nm_{K_2/\bQ}(\nm_{L/K_2}(x))=\nm_{K_1/\bQ}(\nm_{L/K_1}(x))=1
$$
for $x \in T_{L/K_1}$ so the map $\alpha$ is well-defined. 
Also it is trivial that $\ker(\alpha)=T_{L/K_1} \cap T_{L/K_2}$. From the equivalence of categories between the category of tori over $\bQ$ and the category of $G_{\bQ}$-lattices, it is enough to show that
$$
\beta := X^*(\alpha) : X^*(T_{K_2/\bQ}) \rightarrow X^*(T_{L/K_1})
$$
is injective and $\cok(\beta)$ is torsion-free.

Let $\sigma_1, \cdots, \sigma_{n_1}$ be the embeddings of $K_1$ into $\bC$ and $\sigma'_1, \cdots, \sigma'_{n_2}$ be the embeddings of $K_2$ into $\bC$. Since $K_1$ and $K_2$ are linearly disjoint, the embeddings of $L$ into $\bC$ are $\tau_{ij}$ ($1 \leq i \leq n_1$, $1 \leq j \leq n_2$) such that $\tau_{ij} \mid_{K_1} = \sigma_i$ and $\tau_{ij} \mid_{K_2} = \sigma'_j$. By the isomorphism
$$
(\res_{K_2/\bQ} \G_m)_{\overline{\bQ}} 
\cong \prod_{\sigma' : K_2 \hookrightarrow \overline{\bQ}} \G_{m, \overline{\bQ}}^{\sigma'}
\cong \prod_{j=1}^{n_2} \G_{m, \overline{\bQ}}^{\sigma'_{j}}.
$$
($\G_{m, \overline{\bQ}}^{\sigma'}$ are copies of $\G_{m, \overline{\bQ}}$ indexed by $\sigma' : K_2 \hookrightarrow \overline{\bQ}$), the norm map $(\nm_{K_2/\bQ})_{\overline{\bQ}} : (\res_{K_2/\bQ} \G_m)_{\overline{\bQ}} \to \G_{m_,\overline{\bQ}}$ is given by
$$
\prod_{j=1}^{n_2} \G_{m, \overline{\bQ}}^{\sigma'_{j}} \to \G_{m, \overline{\bQ}} \;\; ((x_{\sigma'_j})_{j} \mapsto \prod_{j=1}^{n_2} x_{\sigma'_j}).
$$
Similarly, the norm map $(\nm_{L/K_1})_{\overline{\bQ}} : (\res_{L/\bQ} \G_m)_{\overline{\bQ}} \to (\res_{K_1/
\bQ} \G_m)_{\overline{\bQ}}$ is given by
$$
\prod_{i=1}^{n_1}\prod_{j=1}^{n_2} \G_{m, \overline{\bQ}}^{\tau_{ij}} \to \prod_{i=1}^{n_1} \G_{m, \overline{\bQ}}^{\sigma_{i}} \;\; ((x_{\tau_{ij}})_{i,j} \mapsto ( \prod_{j=1}^{n_2}x_{\tau_{ij}})_{i} ).
$$
These maps imply the isomorphisms
$$
X^*(T_{K_2/\bQ}) 
\cong \cok(X^*(\G_m) \rightarrow X^*(\res_{K_2/\bQ} \G_m))
\cong \cok(\bZ \xrightarrow{x \mapsto (x, \cdots, x)} \prod_{j=1}^{n_2} \bZ_{j})
$$
and
\begin{equation*}
X^*(T_{L/K_1}) 
\cong \cok(X^*(\res_{K_1/\bQ} \G_m) \rightarrow X^*(\res_{L/\bQ} \G_m))
\cong \prod_{i=1}^{n_1} \cok(\bZ_{i} \xrightarrow{x_i \mapsto (x_i, \cdots, x_i)} \prod_{j=1}^{n_2} \bZ_{ij})   
\end{equation*}
as $\bZ$-modules. Here $\bZ_i$, $\bZ_j$, $\bZ_{ij}$ are isomorphic to $\bZ$ and the subscripts are indices. Since the diagram
\[ 
\begin{tikzcd}
T_{L/K_1} \arrow[d, "\nm_{L/K_2}"] \arrow[r] & \res_{L/\bQ} \G_m \arrow[d, "\nm_{L/K_2}"] \\
T_{K_2/\bQ} \arrow[r] & \res_{K_2/\bQ} \G_m
\end{tikzcd}
\]
commutes, the map $\beta$ is induced by the map
$$
\prod_{j=1}^{n_2} (\bZ_j \xrightarrow{x_j \mapsto (x_j, \cdots, x_j)} \prod_{i=1}^{n_1}\bZ_{ij}).
$$
For $A := \bZ^{n_2}/(1, \cdots, 1)\bZ \cong \bZ^{n_2-1}$, the map $\beta$ is given by
$$
\beta : A \rightarrow A^{n_1} \,\, (a \mapsto (a, \cdots, a)).
$$
It is obviously injective and has a torsion-free cokernel. 
\end{proof}

Now the exact sequence \eqref{eq22b} gives
\begin{equation*} 
\begin{split}
\dim (T_{L/K_1} \cap T_{L/K_2}) 
& = \dim T_{L/K_1} - \dim T_{K_2/\bQ} \\
& = ([L:\bQ]-[K_1:\bQ])-([K_2:\bQ]-1) \\
& = ([K_1:\bQ]-1)([K_2:\bQ]-1).
\end{split}
\end{equation*}
Also the equation \eqref{eq22a} and Lemma \ref{lem22b} imply that
\begin{equation} \label{eq22c}
C(T_{L/K_1} \cap T_{L/K_2}) 
= \frac{C(T_{L/K_1})}{C(T_{K_2/\bQ})}
= \frac{D_L}{D_{K_1}D_{K_2}}.
\end{equation}

\subsection{Counting number fields by conductor} \label{Sub23}

In this section, we summarize the known results on counting number fields by conductor. For an integer $n \geq 2$ and a subgroup $G \leq S_n$, denote by $\NF_n(G)$ the set of the isomorphism classes of degree $n$ number fields whose Galois closure has a Galois group permutation-isomorphic to $G$. Denote $\NF_2(C_2)$ by $\NF_2$. (Recall that $C_m$ denotes the cyclic group of order $m$.) For an abelian number field $M$, denote its conductor by $\cond(M)$. First we provide an asymptotics of the number of abelian number fields ordered by conductor.

\begin{proposition} \label{prop23b}
(\cite[Theorems 4 and 5]{Mak93}) For a finite abelian group $G$, the number $N^{\text{con}}(X; G)$ of abelian number fields $L$ such that $\Gal(L/\bQ) \cong G$ and the conductor of $L$ is at most $X$ is given by
$$
N^{\text{con}}(X; G) \sim c_G X(\log X)^{d(G)}
$$
for some constant $c_G>0$ and a nonnegative integer $d(G)$.
\end{proposition}

See \cite[Theorems 4 and 5]{Mak93} for the precise formulas of $c_G$ and $d(G)$. When $G$ is a finite cyclic group (so $G \cong \prod_{i=1}^{r} \bZ / p_i^{e_i} \bZ$ for distinct primes $p_1, \ldots, p_r$ and positive integers $e_1, \ldots, e_r$), then 
$$d(G) = \prod_{i=1}^{r} (e_i+1) - 2.
$$
In particular, $d(C_4) = 1$ and $d(C_6) = 2$.

Let $L_4 \in \NF_4(C_4)$ and denote its quadratic subfield by $L_2$. By the conductor-discriminant formula \cite[VII.11.9 and VII.11.10]{Neu99}, the conductor of $L_4$ is given by $\displaystyle \cond(L_4) = \left ( \frac{D_{L_4}}{D_{L_2}} \right )^{\frac{1}{2}}$. Similarly, if $L_6 \in \NF_6(C_6)$ and $L_i$ ($i=2, 3$) is a subfield of $L_6$ of degree $i$ then the conductor of $L_6$ is given by $\displaystyle \cond(L_6) = \left ( \frac{D_{L_6}}{D_{L_2}D_{L_3}} \right )^{\frac{1}{2}}$. By the above proposition, we have
\begin{equation} \label{eq23a}
\# \left \{ L_4 \in \NF_4(C_4) : \frac{D_{L_4}}{D_{L_2}} \leq X \right \} \sim c_1X^{\frac{1}{2}} \log X
\end{equation}
and
\begin{equation} \label{eq23b}
\# \left \{ L_6 \in \NF_6(C_6) : \frac{D_{L_6}}{D_{L_2}D_{L_3}} \leq X \right \} \sim c_2X^{\frac{1}{2}} (\log X)^2
\end{equation}
for some constants $c_1, c_2>0$.

We have a similar result for some non-abelian number fields. Altuğ, Shankar, Varma and Wilson \cite{ASVW21} counted the number of $D_4$-quartic fields ordered by conductor using both analytic techniques and geometry-of-numbers methods. Let $L \in \NF_4(D_4)$, $M$ be the Galois closure of $L/\bQ$ and $K$ be the unique quadratic subfield of $L$. The conductor of $L$ is defined as the Artin conductor of a Galois representation $\rho_M$ (see \cite[Section 1]{ASVW21}), which is equal to $\frac{D_L}{D_K}$ by the proof of \cite[Proposition 2.4]{ASVW21}. 

\begin{proposition} \label{prop23c}
(\cite[Theorem 1]{ASVW21}) For $L \in \NF_4(D_4)$, denote its unique quadratic subfield by $K$. Then
    $$
    \# \left \{ L \in \NF_4(D_4) : \frac{D_L}{D_K} \leq X \right \} = c(D_4) X \log X + O(X \log \log X)
    $$
    for a constant $\displaystyle c(D_4) := \frac{3}{4} \prod_{p} \left ( 1-\frac{1}{p^2} - \frac{2}{p^3} + \frac{2}{p^4} \right ) >0$ (product over all primes). 
\end{proposition}

In the above proposition, the constant $\frac{3}{4}$ is the sum of $\frac{1}{8}$, $\frac{1}{4}$ and $\frac{3}{8}$ that appears in \cite[Theorem 1]{ASVW21}.

\subsection{Discriminant of a compositum of number fields} \label{Sub25}

Following the exposition of \cite[Section 2]{Wan21}, we give a description of the discriminant of a compositum of two number fields. Let $M$ be a degree $n$ number field and $p$ be a prime. Denote the Galois closure of $M$ by $M^c$ and the inertia group of $M$ at $p$ by $I_{M, p}$. For a nonzero integer $m$, denote the exponent of $p$ in $m$ by $v_p(m)$.

Assume that $M$ is tamely ramified at $p$. Then the inertia group $I_{M, p}$ is cyclic so we can choose its generator $g_{M, p}$. Since $g_{M, p} \in I_{M, p} \subset \Gal(M^c/\bQ) \subset S_n$, we can define the index
$$
\ind(g_{M, p}) := n - \text{the number of orbits of } g_{M, p} \text{ on } \left \{ 1, 2, \cdots, n \right \}
$$
which satisfies the formula $v_p(D_M) = \ind(g_{M, p})$.

\begin{proposition} \label{prop25a}
(\cite[Theorem 2.2 and 2.3]{Wan21}) Let $K_1$ and $K_2$ be number fields such that $K_1^c \cap K_2^c = \bQ$ and $p$ be a prime such that both of $K_1$ and $K_2$ are tamely ramified at $p$. Suppose that $g_{K_1, p} = \prod_{k} c_k$ (product of disjoint cycles) and $g_{K_2, p}=\prod_{l} d_l$. Then
$$
v_p(D_{K_1K_2}) = m_1m_2 - \sum_{k, l} \gcd ( \left | c_{k} \right |,  \left | d_l \right | ),
$$
where $m_i$ is the degree of $K_i$ and $\left | c \right |$ denotes the length of the cycle $c$. If the least common multiple of $\left | c_{k} \right |$ and the least common multiple of $\left | d_l \right |$ are coprime, then we have
$$
v_p(D_{K_1K_2}) 
=v_p(D_{K_1}) \cdot m_2 + v_p(D_{K_2}) \cdot m_1 - v_p(D_{K_1})v_p(D_{K_2}).
$$
\end{proposition}

We also introduce a lemma which concerns the product distribution appearing in counting number fields. It is useful when we consider the compositum of two linearly disjoint number fields.

\begin{proposition} \label{prop24b}
Let $F_i(X) = \# \left \{ s \in S_i : s \leq X  \right \}$ ($i= 1, 2$) be the asymptotic distribution of some multi-set $S_i$ consisting of a sequence of elements of $\R_{\geq 1}$. Suppose that $F_i(X) \sim A_i X^{n_i} (\log X)^{r_i}$ for $n_i>0$, $r_i \in \bZ_{\ge 0}$ and define the product distribution
$$
P(X) := \# \left \{ (s_1, s_2) \in S_1 \times S_2 : s_1s_2 \leq X  \right \}.
$$
\begin{enumerate}
    \item \label{prop24b1} (\cite[Lemma 3.1]{Wan21}) If $n_1=n_2=n$, then 
    $$
    P(X) \sim A_1A_2 \frac{r_1 ! r_2 !}{(r_1+r_2+1)!}nX^n (\log X)^{r_1+r_2+1}.
    $$
    \item \label{prop24b2} (\cite[Lemma 3.2]{Wan21}) If $n_1 > n_2$, then there exists a constant $C>0$ such that
    $$
    P(X) \sim C X^{n_1} (\log X)^{r_1}.
    $$
\end{enumerate}
\end{proposition}

\section{Malle's conjecture for tori over \texorpdfstring{$\bQ$}{Q}} \label{Sec3}

In this section, we provide analogues of Linnik's and Malle's conjectures for tori over $\bQ$ and study the relation between them. First we give a conjecture on the number of the isomorphism classes of $n$-dimensional tori over $\bQ$ counted by Artin conductor.

\begin{conjecture} \label{conj3a}
For every $n \geq 1$, there exists a constant $c_n>0$ satisfying
\begin{equation}
N_n^{\tor}(X) \sim c_n X (\log X)^{n-1}.
\end{equation}
\end{conjecture}

We have the following simple comments on this conjecture.

\begin{enumerate}[label=(\alph*)]
    \item The conjecture is true for $n=1$. Every one-dimensional torus over $\bQ$ is $\G_m$ or $T_{L/\bQ}$ for a quadratic field $L$ (see Example 6 of \cite[Section 4.9]{Vos98}). Since $C(T_{L/\bQ}) = D_L$ by the equation \eqref{eq22a}, we have
    $$
    N_1^{\tor}(X) = N_2(X) + O(1) = \frac{6}{\pi^2}X + O(X^{\frac{1}{2}})
    $$
    by \cite[Section 2.1]{CDO06}. The case $n=2$ will be discussed in Section \ref{Sec4}.
    
    \item It is easy to prove that $N_n^{\tor}(X) \gg X (\log X)^{n-1}$ for each $n \geq 1$. Denote
    \begin{equation*}
    S_1 := \left \{ (L_1, \cdots, L_n) \in \NF_2^n  : \prod_{i=1}^{n} D_{L_i} \leq X \right \} 
    \end{equation*}
    and let $S_2$ be the set of the isomorphism classes of tori $T$ over $\bQ$ which are isomorphic to $\displaystyle \prod_{i=1}^{n} T_{L_i/ \bQ}$ for some quadratic fields $L_1, \cdots, L_n$ and $\displaystyle C(T) = \prod_{i=1}^{n} D_{L_i} \leq X$. One can prove that
    $$
    \left | S_1 \right | \sim \frac{1}{(n-1)!}\left ( \frac{6}{\pi^2} \right )^n X (\log X)^{n-1}
    $$
    by induction on $n$ using Proposition \ref{prop24b}. The map $S_1 \rightarrow S_2$ defined by
    $$
    (L_1, \cdots, L_n) \mapsto \prod_{i=1}^{n} T_{L_i/ \bQ}
    $$
    is surjective and the size of each fiber of the map is at most $n!$. Therefore 
    $$ 
    N_n^{\tor}(X) \geq \left | S_2 \right | \geq \frac{\left | S_1 \right |}{n!} \gg X (\log X)^{n-1}.
    $$
\end{enumerate}

Now we introduce an analogue of Malle's conjecture for algebraic tori over $\bQ$. It follows from a more general conjecture of Ellenberg and Venkatesh \cite{EV05} as explained in the paragraph after Conjecture \ref{conj3b}, so it is not new. 
The important point is that this conjecture implies the above conjecture on the asymptotics of $N_n^{\tor}(X)$ (see Corollary \ref{cor3f}).

Let $n$ be a positive integer and $H \ne 1$ be a finite subgroup of $\GL_n(\bZ)$. As analogues of the numbers $a(G)$ and $b(K,G)$ appear in Malle's conjecture, we define the numbers $a(H)$ and $b(H)$ as follows. First, define $a(H)$ by
\begin{equation}
a(H) := \min_{h \in H \setminus \left \{ I_n \right \}} \rank(h - I_n),
\end{equation}
where $I_n \in \GL_n(\bZ)$ is the identity matrix.
Let $\chi : G_{\bQ} \to (\bZ / |H|\bZ)^{\times}$ be the mod-$|H|$ cyclotomic character, and define an action of $G_{\bQ}$ on $H$ by
$$
\sigma(h) := h^{\chi(\sigma)}
$$
for every $h \in H$ and $\sigma \in G_{\bQ}$. This induces an action of $G_{\bQ}$ on the set $\Cl(H)$ of conjugacy classes of $H$.

\begin{lemma} \label{lem3_rank}
\begin{enumerate}
    \item If $h, h' \in H$ are in the same conjugacy class of $H$, then $\rank(h-I_n) = \rank(h'-I_n)$.

    \item For every $h \in H$ and $\sigma \in G_{\bQ}$, $\rank(\sigma(h) - I_n) = \rank(h-I_n)$.
\end{enumerate}
\end{lemma}

\begin{proof}
\begin{enumerate}
    \item Choose $h_0 \in H$ such that $h=h_0h'h_0^{-1}$. Then $h-I_n = h_0(h'-I_n)h_0^{-1}$ so $\rank(h-I_n) = \rank(h'-I_n)$.

    \item Since $\sigma(h) - I_n = h^{\chi(\sigma)} - I_n = (h-I_n)M_1$ for some $n \times n$ matrix $M_1$ over $\bZ$, we have 
$$
\rank(\sigma(h) - I_n) \le \rank(h-I_n).
$$
By the same reason, we deduce that
\begin{equation*}
\rank(h - I_n) = \rank(\sigma^{-1}(\sigma(h)) - I_n) \le \rank(\sigma(h) - I_n). \qedhere
\end{equation*}
\end{enumerate}
\end{proof}

By Lemma \ref{lem3_rank}(1), the rank of $h-I_n$ is constant within each conjugacy class. By Lemma \ref{lem3_rank}(2), we have $\rank(h_1-I_n)=\rank(h_2-I_n)$ if the conjugacy classes of $h_1$ and $h_2$ in $H$ lie in the same orbit under the action of $G_{\bQ}$ on $\Cl(H)$.

Now let $b(H)$ be the number of orbits $\CC$ of the action of $G_{\bQ}$ on $\Cl(H)$ such that $\rank(h - I_n) = a(H)$ for some (equivalently, all) $h$ contained in a conjugacy class belonging to $\CC$. Then we have the following conjecture, which is an analogue of Malle's conjecture for tori over $\bQ$.

\begin{conjecture} \label{conj3b}
For every $n \geq 1$ and a finite subgroup $1 \neq H \leq \GL_n(\bZ)$, 
\begin{equation}
N_n^{\tor}(X; H) \sim c_H X^{\frac{1}{a(H)}} (\log X)^{b(H)-1}
\end{equation}
where $a(H)$ and $b(H)$ are the positive integers defined as above, and $c_H>0$ is a constant depending only on $H$.
\end{conjecture}

The above conjecture is a direct consequence of Malle's conjecture with modified weights suggested by Ellenberg and Venkatesh \cite{EV05}. Indeed, assume that $T$ is an $n$-dimensional torus over $\bQ$ such that $G_T$ is conjugate to $H$. Then $H$ acts on $V = X^*(T) \otimes \bQ$ via the inclusion $H \le \GL_n(\bZ) \cong \Aut(X^*(T))$. For $h \in H$, define a function $f : \Cl(H) \to \bZ_{\ge 0}$ by $f([h]) := \text{codim} V^h = \rank(h-I_n)$ \cite[Example 4.4]{EV05}, which is well-defined by Lemma \ref{lem3_rank}(1).
Let $a(f)$ and $b_{\bQ}(f)$ be defined as in \cite[Section 4.2]{EV05}.
Then 
$$
a(f) = \max_{h \in H \setminus \{ 1 \}} f([h])^{-1} = a(H)^{-1}
$$
and $b_{\bQ}(f)$ is the number of $G_{\bQ}$-orbits on the set $\{ [h] \in \Cl(H) : f([h])=a(f)^{-1} \}$, which is equal to $b(H)$. Hence \cite[Question 4.3]{EV05} implies Conjecture \ref{conj3b}.

We provide two remarks on Malle's conjecture for tori over $\bQ$.

\begin{remark} \label{rmk3c}
Let $\mathcal{P}_n$ be the group of the permutation matrices in $\GL_n(\bZ)$. There is a canonical isomorphism $\mathcal{P}_n \cong S_n$ which maps $h \in \mathcal{P}_n$ to $\sigma \in S_n$ such that $h_{ij}=1$ if and only if $i=\sigma (j)$. Let $H$ be the subgroup of $\mathcal{P}_n$ which corresponds to a transitive subgroup $G \leq S_n$. Also let $T$ be a torus over $\bQ$ whose splitting field is $L$ and $G_T$ is conjugate to $H$. 

Following the argument of Section \ref{Sub21}, $T$ is determined by $v_1, \cdots, v_n \in \res_{L/\bQ} \G_m$ such that 
$$
h \cdot v_i = \prod_{j=1}^{n} v_j^{h_{ji}} = v_{\sigma(i)}
$$
for all $h \in H$ corresponding to $\sigma \in G$. By the transitivity of $G$, $v_2, \cdots, v_n$ are determined by $v_1$ and $v_1$ is fixed by every element of $H_1 := \left \{ h \in H : h_{11}=1 \right \}$. Since $H_1$ is an index $n$ subgroup of $H \cong \Gal(L/\bQ)$ and $H_1$ has no nontrivial subgroup which is normal in $H$, it corresponds to a Galois group $\Gal(L/K)$ for a degree $n$ number field $K$ such that $K^c=L$. 
(The fact that $H_1$ has no nontrivial subgroup normal in $H$ follows from a purely group-theoretic argument. First note that $H_1 \le H$ corresponds to $G_1 := \{ \sigma \in G : \sigma(1)=1 \} \le G$. If $N \le G_1$ is normal in $G$, then $N = \tau N \tau^{-1} \le \tau G_1 \tau^{-1} = \{ \sigma \in G : \sigma(\tau(1))=\tau(1) \}$ for every $\tau \in G$. Since $G \le S_n$ is transitive, an element $\sigma \in N$ satisfies $\sigma(k)=k$ for every $1 \le k \le n$ so $N$ is trivial.)

Now $T$ is determined by $v_1$ and $h \cdot v_1 = v_1$ for every $h \in H_1$ so
$$
T \subseteq \{ v_1 \in \res_{L/\bQ} \G_m : h \cdot v_1 = v_1 \text{ for every } h \in \Gal(L/K) \} = \res_{K/\bQ} \G_m
$$
for a degree $n$ number field $K$. Since $\dim T = \dim \res_{K/\bQ} \G_m = n$, the inclusion should be an equality. We conclude that
$$
N_n^{\tor}(X; H) = N_n(X; G).
$$

For every $h \in H$ corresponding to $\sigma \in G$, we have
$$
\rank (h - I_n) = n - \text{the number of orbits of } \sigma \text{ on } \left \{ 1, 2, \cdots, n \right \}
= \ind(\sigma)
$$
so $a(H)=a(G)$. The equality $b(H)=b(G)$ is trivial. This shows that Conjecture \ref{conj3b} is a generalization of Malle's conjecture for number fields.
\end{remark}

The next remark shows that the conjecture is compatible with the direct product of $H$.

\begin{remark} \label{rmk3d}
Suppose that the conjecture is true for $1 \neq H_i \leq \GL_{n_i}(\bZ)$ ($i=1, 2$) and define
    $$
    H := \left \{ d(h_1, h_2):= \begin{pmatrix}
h_1 & O \\ 
O & h_2
\end{pmatrix} \in \GL_{n_1+n_2}(\bZ) : h_i \in H_i \right \}.
    $$
    The formula $\rank (d(h_1, h_2) - I_{n_1+n_2}) = \rank(h_1 - I_{n_1}) + \rank(h_2 - I_{n_2})$ implies that
    $$
    a(H) 
    = \min_{(h_1, h_2) \neq (I_{n_1}, I_{n_2})} (\rank(h_1 - I_{n_1}) + \rank(h_2 - I_{n_2}))
    = \min (a(H_1), a(H_2)).
    $$
    Denote by $\CC_i$ a conjugacy class of $H_i$ and $\widetilde{\CC_i}$ a $G_{\bQ}$-orbit on conjugacy classes of $H_i$ (via the cyclotomic character) containing $\CC_i$. 
    The conjugacy classes of $H$ are of the form 
    $$
    \CC = d(\CC_1, \CC_2) := \left \{ d(h_1, h_2) : h_i \in \CC_i \right \},
    $$
    the $G_{\bQ}$-orbit containing $d(\CC_1, I_{n_2})$ is $d(\widetilde{\CC_1}, I_{n_2})$ and the $G_{\bQ}$-orbit containing $d(I_{n_1}, \CC_2)$ is $d(I_{n_1}, \widetilde{\CC_2})$. Therefore the number $b(H)$ is given by
    $$
    b(H) = \left\{\begin{matrix}
b(H_1)+b(H_2) & (a(H_1)=a(H_2)) \\ 
b(H_2) & (a(H_1)>a(H_2)) \\ 
b(H_1) & (a(H_1)<a(H_2))
\end{matrix}\right. .
    $$
    
    If $T$ is a torus over $\bQ$ such that $G_T$ is conjugate to $H$ in $\GL_{n_1+n_2}(\bZ)$, then $T \cong T_1 \times T_2$ where $T_i$ is an $n_i$-dimensional torus such that $G_{T_i}$ is conjugate to $H_i$ in $\GL_{n_i}(\bZ)$. We also have $C(T)=C(T_1)C(T_2)$. 
    Therefore $N_{n_1+n_2}^{\tor}(X; H)$ is bounded above by the product distribution $P(X)$ of $N_{n_1}^{\tor}(X; H_1)$ and $N_{n_2}^{\tor}(X; H_2)$. By Proposition \ref{prop24b}, we have
    $$
    P(X) \sim CX^{\frac{1}{a(H)}} (\log X)^{b(H)-1}
    $$
    for some $C>0$ depending only on $H_1$ and $H_2$. This implies that
    $$
    N_{n_1+n_2}^{\tor}(X; H) \ll_{H} X^{\frac{1}{a(H)}} (\log X)^{b(H)-1}.
    $$
\end{remark}

The next proposition is an analogue of \cite[Lemma 2.2]{Mal04}. Our proof is elementary, but it is not as simple as the proof of \cite[Lemma 2.2]{Mal04}.

\begin{proposition} \label{prop3e}
Let $H \neq 1$ be a finite subgroup of $\GL_n(\bZ)$ with $a(H)=1$. Then $b(H) \leq n$.
\end{proposition}

\begin{proof}
Assume that $n \geq 2$ and denote $I_n$ by $I$ for simplicity. Let $h \neq I$ be an element of $H$ so $h^m=I$ for some $m > 1$. If $\rank (h - I)=1$, then $1$ is an eigenvalue of $h$ with multiplicity $n-1$. If an eigenvalue of $h$ which is not $1$ is $\lambda$, then $\det(h) = 1^{n-1} \lambda = \lambda$. By the condition $h^m=I$, $\det(h) = \pm 1$ and $\lambda \ne 1$ so $\lambda = -1$. Now the condition $\rank(h-I)=1$ implies that $h=I+vw^T$ for some nonzero column vectors $v, w \in \bZ^n$. In this case, $(h-I)^2=(vw^T)^2=c(h-I)$ for $c = w^T v \in \bZ$ so the minimal polynomial of $h$ has degree $2$. Since the eigenvalues of $h$ are $\pm 1$, we have $(h-I)(h+I)=h^2-I=O$ so $h \in H$ has order $2$.

Assume that $a(H)=1$ and $b(H) \geq n+1$. Then there are $h_1, \cdots, h_{n+1} \in H$ which are not conjugate in $H$ and $\rank (h_i-I)=1$ for each $i$. Denote the order of $h \in H$ by $\ord(h)$.
\begin{itemize}
    \item Let $\left \{ x,y \right \} \subset \left \{ h_1, \cdots, h_{n+1} \right \}$ and assume that $\ord(xy)=2k+1$ for $k \in \bZ_{>0}$. Then $y=(xy)^kx(xy)^{-k}$ so $x$ and $y$ are conjugate, which is impossible. Therefore the order of $xy$ is even. 
    
    \item Let $\left \{ x,y,z \right \} \subset \left \{ h_1, \cdots, h_{n+1} \right \}$ and assume that $\ord(xy), \, \ord(xz) \geq 4$. Then $\left \langle x, y, z \right \rangle$ is an infinite group by the classification of finite Coxeter groups \cite{Cox35}, which contradicts the fact that $H$ is finite.
    Therefore at least one of $xy$ or $xz$ has order $2$.
\end{itemize}

For every $i$, there is at most one $j \neq i$ such that $\ord(h_ih_j)>2$. Note that if $\ord(h_ih_j)=2$, then $h_i$ and $h_j$ commute. We may assume that $\ord(h_ih_j)>2$ if and only if $\left \{ i,j \right \}=\left \{ 2t-1, 2t \right \}$ for some $1 \leq t \leq m$, where $m$ is an integer such that $\displaystyle 0 \leq m \leq \frac{n+1}{2}$. Now let $\ord(h_{2t-1}h_{2t})=2 \alpha_t$ ($1 \leq t \leq m$) for positive integers $\alpha_1, \cdots, \alpha_m>1$ and consider the set
$$
S := \left \{ h_1, h_1(h_1h_2)^{\alpha_1}, \cdots, h_{2m-1}, h_{2m-1}(h_{2m-1}h_{2m})^{\alpha_m}, h_{2m+1}, h_{2m+2}, \cdots, h_{n+1} \right \}.
$$
Since $u_t := h_{2t-1}(h_{2t-1}h_{2t})^{\alpha_t}$ is conjugate to one of $h_{2t-1}$ or $h_{2t}$ and $u_t \neq h_{2t-1}$, we have $\left | S \right |=n+1$ and $\rank(h-I)=1$ for every $h \in S$. Also every element of $S$ has order $2$ and the elements of $S$ are pairwise commuting. (The only nontrivial case is the commutativity of $h_{2t-1}$ and $u_t$, which follows from a simple computation.)
Hence the elements of $S$ are simultaneously diagonalizable, i.e. there exists $g \in \GL_n(\bC)$ such that $gSg^{-1} \subset \mathcal{D}_n$ for the group $\mathcal{D}_n$ of diagonal matrices in $\GL_n(\bZ)$. Now we have
$$
\# \left \{ h \in \mathcal{D}_n : \rank(h-I)=1 \right \} = n < \left | gSg^{-1} \right |,
$$
which is a contradiction. 
\end{proof}

\begin{corollary} \label{cor3f}
Conjecture \ref{conj3b} implies Conjecture \ref{conj3a}.
\end{corollary}

\begin{proof}
We have
$$
N_n^{\tor}(X) = \sum_{H} N_n^{\tor}(X; H) + O(1), 
$$
where $H$ runs through a set of representatives for the conjugacy classes of finite nontrivial subgroups of $\GL_n(\bZ)$. This is a finite sum because $\GL_n(\bZ)$ has only finitely many conjugacy classes of finite subgroups; see \cite[Corollary 4.8]{Bor19}. This finiteness result goes back to the classical work of Minkowski \cite{Min87}. The term $O(1)$ accounts for the split torus $\mathbb{G}^n_m$, so it is equal to $1$ if $X \ge 1$, and $0$ otherwise.

Now the corollary follows from Proposition \ref{prop3e} and the fact that $a(D_n)=1$ and $b(D_n)=n$ for the group $D_n$ of diagonal matrices in $\GL_n(\bZ)$. 
\end{proof}


\section{Counting algebraic tori over \texorpdfstring{$\bQ$}{Q} of dimension \texorpdfstring{$2$}{2}} \label{Sec4}

\subsection{Classification of \texorpdfstring{$2$}{2}-dimensional tori over \texorpdfstring{$\bQ$}{Q}} \label{Sub41}

In this section, we give a classification of $2$-dimensional tori over $\bQ$ and compute their Artin conductors. The classification can be done as in Example \ref{ex21a} and it can be also found in the paper of Voskresenskiĭ \cite{Vos65}. In the course of the computation of the Artin conductor $C(T)$, Proposition \ref{prop22a} and the formulas \eqref{eq22a} and \eqref{eq22c} will be repeatedly used.

In each case, denote the splitting field of a torus $T$ by $L$ and identify $\Gal(L/\bQ)$ with $G_T$. For simplicity, denote $D_i := D_{L_i}$, $D_{i'} :=D_{L_i'}$, and so on.

The following list gives the classification of $2$-dimensional tori over $\bQ$ (except for the trivial case $\G_m^2$), together with their Artin conductors. Since there are $12$ conjugacy classes of finite nontrivial subgroups of $\GL_2(\bZ)$, the classification gives $12$ types of $2$-dimensional tori over $\bQ$.
\begin{enumerate}[label=(\roman*)]
\item \label{tor2-1} $G_T \cong C_2$ : $T$ is one of the following types. 

\begin{enumerate}[label=(\alph*)]
        \item \label{tor2A} $\displaystyle G_T = H_{2, A} := \left \langle \bigl(\begin{smallmatrix}
-1 & 0\\ 
0 & -1
\end{smallmatrix}\bigr) \right \rangle$ : $T = T_{L/ \bQ}^2$ and $C(T)=D_L^2$. 

        \item \label{tor2B} $\displaystyle G_T = H_{2, B} := \left \langle \bigl(\begin{smallmatrix}
1 & 0\\ 
0 & -1
\end{smallmatrix}\bigr) \right \rangle$ : $T =  \G_m \times T_{L/ \bQ}$ and $C(T)=D_L$.

        \item \label{tor2C} $\displaystyle G_T = H_{2, C} := \left \langle \bigl(\begin{smallmatrix}
0 & 1 \\ 
1 & 0
\end{smallmatrix}\bigr) \right \rangle$ : $T = \res_{L/ \bQ} \G_m$ and $C(T)=D_L$.
    \end{enumerate}
    
    Note that $\G_m \times T_{L/ \bQ}$ and $\res_{L/ \bQ} \G_m$ are isogenous, but not isomorphic. This corresponds to the fact that $H_{2,B}$ and $H_{2, C}$ are conjugate in $\GL_2(\bQ)$, but not conjugate in $\GL_2(\bZ)$.

\item \label{tor2-2} $G_T \cong C_3$ : $T$ is the following type.

\begin{enumerate}[label=(\alph*)]

    \item \label{tor3A} $\displaystyle G_T = H_{3, A} := \left \langle \bigl(\begin{smallmatrix}
0 & -1\\ 
1 & -1
\end{smallmatrix}\bigr) \right \rangle$ : $T = T_{L/ \bQ}$ and $C(T) = D_L$. 

\end{enumerate}

\item \label{tor2-3} $G_T=\left \langle g \right \rangle \cong C_4$ : $L_4 = L$ has a unique quadratic subfield $L_2 = L^{g^2}$. $T$ is the following type.

\begin{enumerate}[label=(\alph*)]
\item \label{tor4A} $\displaystyle G_T = H_{4, A} := \left \langle \bigl(\begin{smallmatrix}
0 & 1\\ 
-1 & 0
\end{smallmatrix}\bigr) \right \rangle$ : $T=T_{L_4/L_2}$ and $\displaystyle C(T) = \frac{D_4}{D_2}$.

\end{enumerate}

\item \label{tor2-4} $G_T=\left \langle g,h \right \rangle \cong C_2 \times C_2$ : $L$ has $3$ quadratic subfields $L_1=L^g$, $L_2=L^h$ and $L_3=L^{gh}$. By the conductor-discriminant formula, we have $D_L=D_1D_2D_3$. $T$ is one of the following types. (Note that the index starts from (b) because the case $H_{4,A}$ appears earlier.)

\begin{enumerate}[label=(\alph*), start=2]
\item \label{tor4B} $G_T = H_{4, B} := \left \langle \bigl(\begin{smallmatrix}
-1 & 0\\ 
0 & -1
\end{smallmatrix}\bigr), \bigl(\begin{smallmatrix}
0 & 1\\ 
1 & 0
\end{smallmatrix}\bigr) \right \rangle$ : $T = T_{L/L_1}$ and $\displaystyle C(T) = \frac{D_L}{D_1}=D_2D_3$.
        
\item \label{tor4C} $G_T = H_{4, C} := \left \langle \bigl(\begin{smallmatrix}
1 & 0\\ 
0 & -1
\end{smallmatrix}\bigr), \bigl(\begin{smallmatrix}
-1 & 0\\ 
0 & 1
\end{smallmatrix}\bigr) \right \rangle$ : $T = T_{L_1/\bQ} \times T_{L_2/\bQ}$ and $C(T)=D_1D_2$. 
\end{enumerate}

\item\leavevmode \label{tor2-5} $G_T=\left \langle g \right \rangle \cong C_6$ : $L_6=L$ has a unique cubic subfield $L_3=L^{g^3}$ and a unique quadratic subfield $L_2=L^{g^2}$. $T$ is the following type. 

\begin{enumerate}[label=(\alph*)]
\item \label{tor6A} $\displaystyle G_T = H_{6,A} := \left \langle \bigl(\begin{smallmatrix}
1 & -1\\ 
1 & 0
\end{smallmatrix}\bigr) \right \rangle$ : $T = T_{L_6/L_2} \cap T_{L_6/L_3}$ and $\displaystyle C(T) = \frac{D_6}{D_2D_3}$. 
\end{enumerate}

\item \label{tor2-6} $G_T=\left \langle g, h : g^3=h^2=(gh)^2=1 \right \rangle \cong S_3$ : $L_6 = L$ has $3$ isomorphic cubic subfields $L^h, L^{gh}$ and $L^{g^2h}$, denoted by $L_3$ and a unique quadratic subfield $L_2 = L^g$. By the work of Hasse \cite{Has30}, we have $D_6=D_3^2D_2$. (See also \cite[Theorem 4]{FK03} for its generalization to Frobenius groups. In fact, taking $(k, M, K, N)=(\bQ, L_2, L_3, L_6)$ in \cite[Theorem 4]{FK03} gives
$$
D_3=D_2^{\frac{3-1}{2}} N_{L_2/\bQ}(D_{L_6/L_2})^{\frac{1}{2}} = D_2 \left ( \frac{D_6}{D_2^3} \right )^{\frac{1}{2}} = \frac{D_6^{\frac{1}{2}}}{D_2^{\frac{1}{2}}}
$$ 
so $D_6=D_3^2D_2$.) $T$ is one of the following types. 

\begin{enumerate}[label=(\alph*), start=2]

\item \label{tor6B} $G_T = H_{6, B} := \left \langle \bigl(\begin{smallmatrix}
0 & -1\\ 
1 & -1
\end{smallmatrix}\bigr), \bigl(\begin{smallmatrix}
0 & 1\\ 
1 & 0
\end{smallmatrix}\bigr) \right \rangle$ : Let $g = \bigl(\begin{smallmatrix}
0 & -1\\ 
1 & -1
\end{smallmatrix}\bigr)$, $h=\bigl(\begin{smallmatrix}
0 & 1\\ 
1 & 0
\end{smallmatrix}\bigr)$ and
$$
T_0 := \{ x \in T_{L_6/\bQ} : x \cdot gx \cdot g^2x=1 \} = T_{L_6/L_2}.
$$
(It is same as $T$ in \cite[Section 4, (7-a)]{Vos65}.) Then 
$$
T = \{ v \in T_{L_6/L_2} : hv = v \} = T_{L_3/\bQ}
$$
and $C(T) = D_3$. 
        
\item \label{tor6C} $G_T = H_{6, C} := \left \langle \bigl(\begin{smallmatrix}
0 & -1\\ 
1 & -1
\end{smallmatrix}\bigr), \bigl(\begin{smallmatrix}
0 & -1\\ 
-1 & 0
\end{smallmatrix}\bigr) \right \rangle$ : Let $g = \bigl(\begin{smallmatrix}
0 & -1\\ 
1 & -1
\end{smallmatrix}\bigr)$, $h=\bigl(\begin{smallmatrix}
0 & -1\\ 
-1 & 0
\end{smallmatrix}\bigr)$ and
$$
T_0 := \{ x \in T_{L_6/\bQ} : x \cdot gx \cdot g^2x =1 \} = T_{L_6/L_2}.
$$
(It is same as $T$ in \cite[Section 4, (7-b)]{Vos65}.) Then 
$$
T = \{ v \in T_{L_6/L_2} : v \cdot hv =1 \} = T_{L_6/L_2} \cap T_{L_6/L_3}
$$
and $\displaystyle C(T) = \frac{D_6}{D_2D_3}=D_3$. 
\end{enumerate}

\item \label{tor2-7} $G_T=\left \langle g, h : g^4=h^2=(gh)^2=1 \right \rangle \cong D_4$ : $L_4=L^{gh}$ is a quartic $D_4$-field with a unique quadratic subfield $L_2=L^{\left \langle g^2, gh \right \rangle}$. $T$ is the following type.

\begin{enumerate}[label=(\alph*)]

    \item \label{tor8A} $\displaystyle G_T = H_{8, A} := \left \langle \bigl(\begin{smallmatrix}
0 & 1\\ 
-1 & 0
\end{smallmatrix}\bigr), \bigl(\begin{smallmatrix}
0 & 1\\ 
1 & 0
\end{smallmatrix}\bigr) \right \rangle$ : $T = T_{L_4/L_2}$ and $\displaystyle C(T) = \frac{D_4}{D_2}$. 

\end{enumerate}

\item \label{tor2-8} $G_T=\left \langle g, h : g^6=h^2=(gh)^2=1 \right \rangle \cong D_6$ : $L_6 = L^{g^4h} \in \NF_6(D_6)$ has a unique cubic subfield $L_3=L^{\left \langle g^3, gh \right \rangle} \in \NF_3(S_3)$ and a unique quadratic subfield $L_2 = L^{\left \langle g^2, h \right \rangle}$. $T$ is the following type. 

\begin{enumerate}[label=(\alph*)]

    \item \label{tor12A} $\displaystyle G_T = H_{12, A} := \left \langle \bigl(\begin{smallmatrix}
1 & -1\\ 
1 & 0
\end{smallmatrix}\bigr), \bigl(\begin{smallmatrix}
0 & 1\\ 
1 & 0
\end{smallmatrix}\bigr) \right \rangle$ : $T = T_{L_6/L_3} \cap T_{L_6/L_2}$ and $\displaystyle C(T) = \frac{D_6}{D_3D_2}$. 

\end{enumerate}
\end{enumerate}

\begin{table}[h]
\centering
\begin{tabular}{|c|c|c|c|c|c|}
\hline
$H$ & $a(H)$ & $b(H)$ & $H$ & $a(H)$ & $b(H)$ \\ \hline
$H_{2,A}$ & 2 & 1 & $H_{4,C}$ & 1 & 2 \\ \hline
$H_{2,B}$ & 1 & 1 & $H_{6,A}$ & 2 & 3 \\ \hline
$H_{2,C}$ & 1 & 1 & $H_{6,B}$ & 1 & 1 \\ \hline
$H_{3,A}$ & 2 & 1 & $H_{6,C}$ & 1 & 1 \\ \hline
$H_{4,A}$ & 2 & 2 & $H_{8,A}$ & 1 & 2 \\ \hline
$H_{4,B}$ & 1 & 2 & $H_{12,A}$ & 1 & 2 \\ \hline
\end{tabular}
\vspace{1mm}
\caption{The numbers $a(H)$ and $b(H)$}
\label{Table}
\end{table}

The asymptotics of $N_2^{\tor}(X; H)$ for $H \neq H_{12, A}$ can be easily computed. Therefore the essential new result of Section \ref{Sec4} is the estimation of $N_2^{\tor}(X; H_{12, A})$. 

\begin{proposition} \label{prop41a}
Conjecture \ref{conj3b} holds for every finite nontrivial subgroup of $\GL_2(\bZ)$ which is not conjugate to $H_{12, A}$.
\end{proposition}

\begin{proof}
Since Malle's conjecture is true for abelian number fields and $S_3$-cubic fields, the relation between the Artin conductors $C(T)$ and the corresponding discriminants of number fields implies that Conjecture \ref{conj3b} holds when $H$ is one of 
$$
H_{2, A}, \, H_{2, B}, \, H_{2, C}, \, H_{3, A}, \, H_{6, B} \text{ and } H_{6, C}.
$$
By Proposition \ref{prop24b}, the conjecture is true if $H$ is $H_{4, B}$ or $H_{4, C}$. By the equations \eqref{eq23a} and \eqref{eq23b}, the conjecture is true if $H$ is $H_{4, A}$ or $H_{6, A}$. By Proposition \ref{prop23c}, the conjecture is true for $H=H_{8, A}$.
\end{proof}

\subsection{Asymptotics of \texorpdfstring{$N_2^{\tor}(X; H_{12, A})$}{N2H12A}} \label{Sub42}

Let $L$ be a $D_6$-sextic field with a unique cubic subfield $F \in \NF_3(S_3)$ and a unique quadratic subfield $K$. By \ref{tor2-8} in Section \ref{Sec4}, the number $N_2^{\tor}(X; H_{12, A})$ is equal to the number of $L \in \NF_6(D_6)$ such that 
$$
C(L):= \frac{D_L}{D_F D_K} \leq X.
$$
In this section, we estimate the number of such $L$ based on the strategy of the paper \cite{MTTW20}, where the authors proved Malle's conjecture for $D_6$-sextic fields. First we express $v_p(D_L)$ (the $p$-adic valuation of $D_L$) in terms of $v_p(D_F)$ and $v_p(D_K)$.

\begin{proposition} \label{prop42b}
\begin{equation} \label{eq42a}
C(L) = \frac{D_F D_K^2}{Cm^2}
\end{equation}
where $\displaystyle C = 2^a 3^b$ for $0 \leq a \leq 9$, $0 \leq b \leq 3$ and $m$ is the product of the primes $p>3$ which divides both $D_F$ and $D_K$. 
\end{proposition}

\begin{proof}
Since $L=FK$ and the number fields $F$, $K$ are linearly disjoint,
$$
\lcm (D_F^2, D_K^3) = \lcm (D_F^{[L:F]}, D_K^{[L:K]}) \mid D_L \mid D_F^2 D_K^3.
$$
(The second divisibility follows from the inclusion $\mathcal{O}_F \mathcal{O}_K \subseteq \mathcal{O}_L$ and the fact that the order $\mathcal{O}_F \mathcal{O}_K$ has discriminant $D_F^{[L:F]} D_K^{[L:K]}$.)
Hence $\displaystyle N := \frac{D_F^2 D_K^3}{D_L}$ is a positive integer which divides $\gcd(D_F^2, D_K^3)$. We need to show that $N=Cm^2$.
\begin{itemize}
    \item If $p>3$ and $p \nmid \gcd(D_F, D_K)$, then $v_p(N)=0=v_p(Cm^2)$. 
    
    \item If $p>3$ and $p \mid \gcd(D_F, D_K)$, then $v_p(N)=2=v_p(Cm^2)$ by Proposition \ref{prop25a} (cf. \cite[Table 1]{MTTW20}).
    
    \item If $p \in \left \{ 2,3 \right \}$, then $0 \leq v_p(N) \leq v_p(D_K^3)$ so $v_2(N) \leq 9$ and $v_3(N) \leq 3$. \qedhere
\end{itemize}
\end{proof}

Denote the set of positive squarefree integers by $\Sqf$. For a prime $p$, denote by $\Sqf_p$ the set of positive integers which are squarefree outside $p$. By the above proposition and the isomorphism $D_6 \cong S_3 \times C_2$, 
\begin{equation} \label{eq42b}
\begin{split}
N_2^{\tor}(X; H_{12, A}) & = \# \left \{ L \in \NF_6(D_6) : C(L) \leq X \right \}  \\
& \leq \# \left \{ (F, K) \in \NF_3(S_3) \times \NF_2 : \frac{D_F D_K^2}{m^2} \leq \beta X \right \} \\
 & =: A_1(\beta X)
\end{split}
\end{equation}
for $\beta := 2^9 3^3$. Since the number of quadratic fields $K$ satisfying $m \mid D_K$ and $\displaystyle \frac{D_K}{m} \leq \left ( \frac{X}{D_F} \right )^{\frac{1}{2}}$ for given $F$ and $m$ is at most $\displaystyle 2 \left ( \frac{X}{D_F} \right )^{\frac{1}{2}}$, we have
\begin{equation} \label{eq42c}
A_1(X) \leq \sum_{\substack{m \in \, \Sqf \\ \gcd(m, 6)=1}} \sum_{\substack{F \in \NF_3(S_3) \\ D_F \leq X \\ m \mid D_F}} 2 \left ( \frac{X}{D_F} \right )^{\frac{1}{2}}.
\end{equation}

The Galois closure $F^c$ of $F/\bQ$ has a unique quadratic subfield $E$. (It is called the quadratic resolvent of $F$.) Then $D_F=D_E f^2$ for some $f \in \, \Sqf_3$ by \cite[Theorem 9.2.6]{Coh00}. (Precisely, 
$D_F=D_E \mathfrak{a}_1 \mathfrak{a}_3^2$ by \cite[Theorem 9.2.6(4)]{Coh00}, $\mathfrak{a}_1=\bZ$ by \cite[Theorem 9.2.6(2)]{Coh00} so $D_F=D_Ef^2$ for $f = \mathfrak{a}_3$; moreover, $p^2 \mid \mathfrak{a}_3$ implies $p=3$ by \cite[Theorem 9.2.6(6)]{Coh00}.)

\begin{lemma} \label{lem4_new1}
Let $E$ be a quadratic field and $f \in \Sqf_3$. Then the number of $F \in \NF_3(S_3)$ whose quadratic resolvent is $E$ and satisfying $D_F = D_E f^2$ is
$$ O(h_3(E) \cdot 2^{w(f)}), $$ 
where $h_3(E)$ is the size of the $3$-torsion subgroup of the class group of $E$, $w(f)$ is the number of prime divisors of $f$ and the implied constant is absolute.
\end{lemma}

\begin{proof}
Let $E=\bQ(\sqrt{D})$. Then the number of $F \in \NF_3(S_3)$ whose quadratic resolvent is $E$ and satisfying $D_F = D_E f^2$ is equal to the coefficient of $f^{-s}$ in the series $\Phi_D(s)$ given in \cite[Definition 2.4]{CT14}. By \cite[Theorem 2.5]{CT14}, along with \cite[Definition 2.3 and Table 1]{CT14}, the coefficient of $f^{-s}$ in $\Phi_D(s)$ is 
$$
O(2^{w(f)})+O(\left | \mathcal{L}_3(D)\right | 2^{w(f)}) = O(\left | \mathcal{L}_3(D)\right | 2^{w(f)}).
$$
By \cite[Proposition 3.7]{CT14}, 
$\left | \mathcal{L}_3(D)\right | \le \frac{3h_3(E)-1}{2}$ so the lemma holds.
\end{proof}

Now let
\begin{equation*}
\begin{split}
S_1 & := \left \{ p \mid m : F \text{ is not totally ramified at } p \right \} \\
S_2 & := \left \{ p \mid m : F \text{ is totally ramified at } p \right \} \\
m_i & := \prod_{p \in S_i} p \;\; (i=1, 2).
\end{split}
\end{equation*}
Then $m_1$ and $m_2$ are coprime, squarefree integers such that
$$
m_1 \mid D_E, \, 
m_2 \mid f, \, 
m_1m_2= m \text{ and } \gcd (m_1m_2, 6)=1.
$$
(For each $p \in S_2$, the extension $F/\bQ$ is totally and tamely ramified at $p$ so $v_p(D_F)=[F:\bQ]-1=2 > v_p(D_E)$. Hence $p \mid f$, which implies that $m_2 \mid f$.)
Now the inequality \eqref{eq42c} transforms into
\begin{equation} \label{eq42d}
A_1(X) \ll 
\sum_{\substack{m_1,  m_2 \in \, \Sqf \\ (m_1m_2, 6)=1 \\ (m_1, m_2)=1}}
\sum_{\substack{E \in \NF_2 \\ m_1 \mid D_E}} 
\sum_{\substack{f \in \, \Sqf_3 \\ m_2 \mid f \\ D_E f^2 \leq X}}
\left ( \frac{X}{D_E f^2} \right )^{\frac{1}{2}} h_3(E) \cdot 2^{w(f)}.
\end{equation}

We estimate the right-hand side of the above inequality by summing over the intervals 
$$
D_E f^2 \in [B, 2B)
$$
for $B=2^i$ ($0 \leq i \leq \log_2 X$). The inequality \eqref{eq42d} implies that
\begin{equation} \label{eq42e}
A_1(X) \ll \sum_{i=0}^{\left \lfloor \log_2 X \right \rfloor} A_2(X; 2^i)
\end{equation}
for
\begin{equation} \label{eq42f}
\begin{split}
A_2(X; B) & := \sum_{\substack{m_1, m_2 \in \, \, \Sqf \\ (m_1m_2, 6)=1 \\ (m_1, m_2)=1}}
\sum_{\substack{E \in \NF_2 \\ m_1 \mid D_E}} 
\sum_{\substack{f \in \, \Sqf_3 \\ m_2 \mid f \\  D_E f^2  \leq X \\ B \leq D_E f^2 < 2B}}
\left ( \frac{X}{D_E f^2} \right )^{\frac{1}{2}} h_3(E) \cdot 2^{w(f)} \\
& \leq \frac{X^{\frac{1}{2}}}{B^{\frac{1}{2}}} 
\sum_{\substack{f \in \, \Sqf_3 \\ f < (2B)^{\frac{1}{2}}}} 2^{w(f)}
\sum_{m_2 \mid f} 1
\sum_{\substack{E \in \NF_2 \\ \frac{B}{f^2} \leq D_E < \frac{2B}{f^2}}}  h_3(E)
\sum_{m_1 \mid D_E} 1 \\
& \leq \frac{X^{\frac{1}{2}}}{B^{\frac{1}{2}}} 
\sum_{\substack{f \in \, \Sqf_3 \\ f < (2B)^{\frac{1}{2}} }} 2^{w(f)} \tau(f)
\sum_{\substack{E \in \NF_2 \\  D_E < \frac{2B}{f^2}}}  h_3(E) \tau(D_E).
\end{split}
\end{equation}
Here $\tau(n)$ denotes the number of divisors of $n$. The inequalities \eqref{eq42e} and \eqref{eq42f} show that it is essential to give an upper bound of the function
$$
g(X) := \sum_{\substack{E \in \NF_2 \\  D_E < X}}  h_3(E) \tau(D_E),
$$
which will be done in the next section.

\subsection{Upper bound of \texorpdfstring{$g(X)$}{g(X)}} \label{Sub43}

The estimation of the function $g(X)$ is the key part of Section \ref{Sec4}. In this section, we provide both conditional and unconditional results on the upper bound of $g(X)$. 

\begin{proposition} \label{prop43a}
\begin{equation} \label{eq43a}
g(X) \ll_{\ve} X^{1+\frac{\log 2 + \ve}{\log \log X}} \ll_{\ve} X^{1 + \ve}.
\end{equation}
\end{proposition}

\begin{proof}
A classical result of Wigert \cite{Wig07} (see also \cite[Theorem 317]{HW79}) states that for any $\ve>0$, 
$$
\max_{n < X} \tau(n) < X^{\frac{\log 2 + \ve}{\log \log X}}
$$
for sufficiently large $X$. Therefore
$$
g(X) \ll_{\ve} X^{\frac{\log 2 + \ve}{\log \log X}} \sum_{\substack{E \in \NF_2 \\  D_E < X}}  h_3(E) 
\ll_{\ve} X^{1+\frac{\log 2 + \ve}{\log \log X}} \ll_{\ve} X^{1 + \ve}
$$
by the theorem of Davenport and Heilbronn \cite[Theorem 3]{DH71}.
\end{proof}

The following corollary will be useful for the estimation of $A_2(X; B)$; see the proof of Theorem \ref{thm44a}. 

\begin{corollary} \label{cor43b}
For any $Y \in [1, 2X]$, $g(Y) \ll_{\ve} Y X^{\frac{\log 2 + \ve}{\log \log X}}$ as $X \rightarrow \infty$. 
\end{corollary}

Now we consider the following version of the Cohen--Lenstra heuristics. Denote the set of the isomorphism classes of real (resp. imaginary) quadratic fields by $\NF_2^{+}$ (resp. $\NF_2^{-}$). The elements of $\NF_2^{+}$ and $\NF_2^{-}$ are ordered by the absolute values of their discriminants. 
For a number field $K$ and a prime $p$, denote the size of the $p$-torsion subgroup of the class group of $K$ by $h_p(K)$. 

\begin{conjecture} \label{conj43c}
(Cohen--Lenstra) Let $p$ be an odd prime and $\alpha$ be a positive integer. 
\begin{enumerate}
    \item \label{conj43c1} (\cite[(C10)]{CL84}) The average of $\displaystyle \prod_{0 \leq i < \alpha} (h_p(K)-p^i)$ for $K \in \NF_2^{+}$ is $p^{-\alpha}$.

    \item \label{conj43c2} (\cite[(C6)]{CL84}) The average of $\displaystyle \prod_{0 \leq i < \alpha} (h_p(K)-p^i)$ for $K \in \NF_2^{-}$ is $1$.
\end{enumerate}
\end{conjecture}

By \cite[Theorem 3]{DH71}, the above conjecture is true when $p=3$ and $\alpha=1$. It is the only known case of the conjecture. 

\begin{remark} \label{rmk43d}
If the conjecture is true for a fixed prime $p$ and each of $1 \leq \alpha \leq m$, then the $m$-th moment of the number $h_p(K)$ for quadratic fields $K$ is given by
    \begin{equation*} 
    \sum_{\substack{K \in \NF_2 \\ D_K \leq X}} h_p(K)^m \sim c X
    \end{equation*}
    for some explicit constant $c>0$ depending only on $p$ and $m$. 
\end{remark}

The above conjecture enables us to obtain the following upper bound of $g(X)$, which is much stronger than the upper bound given in Proposition \ref{prop43a}. 

\begin{proposition} \label{prop43e}
Let $m \geq 2$ be an integer. Under the assumption of Conjecture \ref{conj43c} for $p=3$ and each of $1 \leq \alpha \leq m$, we have
\begin{equation} \label{eq43c}
g(X) \ll_m X (\log X)^{2^{\frac{m}{m-1}}-1}.
\end{equation}
In particular, if Conjecture \ref{conj43c} holds for $p=3$ and every $\alpha >0$, then
\begin{equation} \label{eq43d}
g(X) \ll_{\ve} X (\log X)^{1 + \ve}.
\end{equation}
\end{proposition}

\begin{proof}
Assume that Conjecture \ref{conj43c} holds for $p=3$ and $1 \leq \alpha \leq m$. Let 
\begin{equation*}
a_r(X) := \# \left \{ E \in \NF_2 : D_E \leq X \text{ and } h_3(E)=3^r \right \}
\end{equation*}
and denote by $w(k)$ the number of prime divisors of an integer $k$. By the theorem of Hardy and Ramanujan \cite{HR17}, there are constants $c_1, c_2>0$ such that
\begin{equation} \label{eq43e}
\# \left \{ E \in \NF_2 : D_E \leq X \text{ and } w(D_E)=k \right \} < \frac{c_1 X}{\log X}\cdot \frac{(\log \log X + c_2)^{k-1}}{(k-1)!}
\end{equation}
for every $X \geq 2$ and $k \geq 1$.

For a quadratic field $E$, $v_2(D_E) \leq 3$ so $\tau(D_E) \leq 2^{w(D_E)+1}$. Therefore
\begin{equation} \label{eq43f}
\begin{split}
g(X) & \leq \sum_{r=0}^{\infty} \sum_{k=1}^{\infty}
3^r 2^{k+1} \# \left \{ E \in \NF_2 : D_E \leq X, \, h_3(E)=3^r \text{ and } w(D_E)=k \right \} \\
& \leq \sum_{r=0}^{\infty} \sum_{k=1}^{\infty}
3^r 2^{k+1} \min \left \{ a_r(X) , \, \frac{c_1 X}{\log X}\cdot \frac{(\log \log X + c_2)^{k-1}}{(k-1)!} \right \}
\end{split}
\end{equation}
by the inequality \eqref{eq43e}. Suppose that $X \geq 2$ and denote
$$
S(r, k) :=  \min \left \{ a_r(X) , \, \frac{c_1 X}{\log X}\cdot \frac{(\log \log X + c_2)^{k-1}}{(k-1)!} \right \}.
$$
To bound the sum $\displaystyle \sum_{(r,k) \in \bZ_{\geq 0} \times \bZ_{\geq 1}} 3^r 2^{k+1} S(r,k)$, we divide the set $\bZ_{\geq 0} \times \bZ_{\geq 1}$ into two parts:
\begin{equation*} 
\begin{split}
R_{1, m} & := \left \{ (r,k) \in \bZ_{\geq 0} \times \bZ_{\geq 1} : 2^{k-1} < 3^{r(m-1)} \right \} \\
R_{2, m} & := \left \{ (r,k) \in \bZ_{\geq 0} \times \bZ_{\geq 1} : 2^{k-1} \geq 3^{r(m-1)} \right \}.
\end{split}
\end{equation*}
Now we estimate the sum $\displaystyle \sum_{(r,k) \in R_{i, m}} 3^r 2^{k+1} S(r,k)$ for $i= 1, 2$.

\begin{enumerate}[label=(\roman*)]

    \item \label{r1sum} 
    For a given $r \ge 0$, let $k_r$ be the maximal nonnegative integer such that $2^{k_r-1} < 3^{r(m-1)}$. (Note that if $r=0$, then $k_r=0$.) Then
    $$
    \sum_{\substack{k \geq 1 \\ (r,k) \in R_{1, m}} } 3^r 2^{k+1} =
    \sum_{1 \le k \le k_r} 3^r 2^{k+1}
    < 3^r 2^{k_r+2} < 8 (3^r)^m.
    $$ 
    Therefore 
    \begin{equation} \label{eqnew3}
    \sum_{(r,k) \in R_{1, m}} 3^r 2^{k+1} S(r,k) 
    \leq \sum_{r=0}^{\infty} \sum_{\substack{k \geq 1 \\ (r,k) \in R_{1, m}} } 3^r 2^{k+1} a_r(X) 
    \leq 8 \sum_{r=0}^{\infty} (3^r)^m a_r(X)
    \end{equation}
    and
    \begin{equation} \label{eqnew4}
        \sum_{r=0}^{\infty} (3^r)^m a_r(X)
        = \sum_{r=0}^{\infty} \sum_{\substack{E \in \NF_2 \\ D_E \leq X \\ h_3(E)=3^r}} h_3(E)^m 
        = \sum_{\substack{E \in \NF_2 \\ D_E \leq X }} h_3(E)^m \ll_m X
    \end{equation}
    by Remark \ref{rmk43d}. The inequalities \eqref{eqnew3} and \eqref{eqnew4} imply that
    \begin{equation} \label{eqnew5}
    \sum_{(r,k) \in R_{1, m}} 3^r 2^{k+1} S(r,k)  \ll_m X.
    \end{equation}

    \item \label{r2sum} 
    For a given $k \ge 1$, let $r_k$ be the maximal nonnegative integer such that $2^{k-1} \ge 3^{r_k(m-1)}$. Then
    $$
    \sum_{\substack{r \geq 0 \\ (r,k) \in R_{2, m}}} 3^r 2^{k+1} =
    \sum_{0 \le r \le r_k} 3^{r} 2^{k+1}
    < 3^{r_k+1} 2^{k} 
    \le 3 (2^{k-1})^{\frac{1}{m-1}} 2^{k}
    = 6 (2^{k-1})^{\frac{m}{m-1}}.
    $$ 
    Therefore 
    \begin{align*}
    \sum_{(r,k) \in R_{2, m}} 3^r 2^{k+1} S(r,k) 
    & \leq \sum_{k=1}^{\infty} \sum_{\substack{r \geq 0 \\ (r,k) \in R_{2, m}} } 3^r 2^{k+1} S(r,k) \\
    & \leq 6 \sum_{k=1}^{\infty} (2^{k-1})^{\frac{m}{m-1}} \frac{c_1 X}{\log X} \cdot \frac{(\log \log X + c_2)^{k-1}}{(k-1)!} \\
    & \leq 6c_1 \frac{X}{\log X} \sum_{k=1}^{\infty} \frac{(2^{\frac{m}{m-1}}(\log \log X + c_2))^{k-1}}{(k-1)!} \\
    & = 6c_1 \frac{X}{\log X} e^{2^{\frac{m}{m-1}}(\log \log X + c_2)} \\
    & \leq (6c_1 e^{4c_2}) \frac{X}{\log X} (\log X)^{2^{\frac{m}{m-1}}}
    \end{align*}
    so 
    \begin{equation} \label{eqnew6}
    \sum_{(r,k) \in R_{2, m}} 3^r 2^{k+1} S(r,k)  \ll  X (\log X)^{2^{\frac{m}{m-1}}-1}.
    \end{equation}
\end{enumerate}
Now the proposition follows from the inequalities \eqref{eq43f}, \eqref{eqnew5} and \eqref{eqnew6}.
\end{proof}

The following corollary follows from Proposition \ref{prop43e} and the fact that $g(X)=0$ for every $X<3$.

\begin{corollary} \label{cor43f}
Assume that Conjecture \ref{conj43c} holds for $p=3$ and every $\alpha >0$. Then for every $\ve >0$, there exists $c_{\ve}>0$ such that $g(X) \leq c_{\ve} X (\log X)^{1+\ve}$ for every $X \geq 1$.
\end{corollary}

\subsection{Main theorems} \label{Sub44}

In this section, we give asymptotic upper and lower bounds for $N_2^{\tor}(X; H_{12, A})$ and $N_2^{\tor}(X)$. 

\begin{theorem} \label{thm44a}
\begin{subequations} 
\begin{align}
X \ll N_2^{\tor}(X; H_{12, A}) & \ll_{\ve} X^{1+\frac{\log 2 + \ve}{\log \log X}}. \\
X \log X \ll N_2^{\tor}(X) & \ll_{\ve} X^{1+\frac{\log 2 + \ve}{\log \log X}}. 
\end{align}
\end{subequations}
\end{theorem}

\begin{proof}
By Proposition \ref{prop41a}, the upper and lower bounds of $N_2^{\tor}(X)$ follows from the upper and lower bounds of $N_2^{\tor}(X; H_{12, A})$. (Note that $a(H)=1$ and $b(H)=2$ for $H \in \left \{ H_{4,B}, H_{4,C}, H_{8,A} \right \}$.) By the inequality \eqref{eq42f} and Corollary \ref{cor43b}, we have
\begin{equation} \label{eq44a}
\begin{split}
A_2(X; B) & \leq \frac{X^{\frac{1}{2}}}{B^{\frac{1}{2}}} 
\sum_{\substack{f \in \Sqf_3 \\ f < (2B)^{\frac{1}{2}} }} 2^{w(f)} \tau(f)  g(\frac{2B}{f^2}) \\
& \ll_{\ve} \frac{X^{\frac{1}{2}}}{B^{\frac{1}{2}}} 
\sum_{\substack{f \in \Sqf_3 \\ f < (2B)^{\frac{1}{2}} }} 2^{w(f)} \tau(f) \cdot \frac{2B}{f^2} X^{\frac{\log 2 + \ve}{\log \log X}} \\
& \ll_{\ve} X^{\frac{1}{2}+\frac{\log 2 + \ve}{\log \log X}}B^{\frac{1}{2}} 
\cdot \sum_{f} \frac{2^{w(f)} \tau(f)}{f^2} \\
& \ll_{\ve} X^{\frac{1}{2}+\frac{\log 2 + \ve}{\log \log X}}B^{\frac{1}{2}} 
\end{split}
\end{equation}
for $B \leq X$. The last inequality is due to the fact that $2^{w(f)} \tau(f) \ll_{\ve} f^{\ve}$, which implies the convergence of the sum $\displaystyle \sum_{f}
\frac{2^{w(f)} \tau(f)}{f^2}$. Now the inequalities \eqref{eq42b}, \eqref{eq42e} and \eqref{eq44a} imply that
\begin{align*}
N_2^{\tor}(X; H_{12, A}) 
& \leq A_1(\beta X) \\
& \ll \sum_{i=0}^{\left \lfloor \log_2 (\beta X) \right \rfloor} A_2(\beta X; 2^i) \\
& \ll_{\ve} (\beta X)^{\frac{1}{2}+\frac{\log 2 + \ve}{\log \log \beta X}} \sum_{i=0}^{\left \lfloor \log_2 (\beta X) \right \rfloor} 2^{\frac{i}{2}} \\
& \ll_{\ve} X^{1+\frac{\log 2 + \ve}{\log \log X}}.
\end{align*}
A trivial upper bound $C(L) \leq D_F D_K^2$ (by the equation \eqref{eq42a}) implies that
\begin{equation*}
\begin{split}
N_2^{\tor}(X; H_{12, A}) & = \# \left \{ L \in \NF_6(D_6) : C(L) \leq X \right \}  \\
& \geq \# \left \{ (F, K) \in \NF_3(S_3) \times \NF_2 : D_FD_K^2 \leq X  \text{ and } K \nsubseteq F^c \right \} \\
& \geq \# \left \{ F \in \NF_3(S_3)  : 16 D_F \leq X \text{ and } \bQ(i) \nsubseteq F^c \right \} \\
& + \# \left \{ F \in \NF_3(S_3)  : 9 D_F \leq X  \text{ and } \bQ(\sqrt{-3}) \nsubseteq F^c \right \} \\
& \geq \# \left \{ F \in \NF_3(S_3)  : 16 D_F \leq X  \right \} \\
& \gg X. \qedhere
\end{split}
\end{equation*}
\end{proof}

We have $X (\log X)^N \ll_{\ve} X^{1+\frac{\log 2 + \ve}{\log \log X}} \ll_{\ve} X^{1+\ve}$ for every positive integer $N$. Under the assumption of the Cohen--Lenstra heuristics for $p=3$, we obtain a much better upper bound. 

\begin{theorem} \label{thm44b}
Assume that Conjecture \ref{conj43c} holds for $p=3$ and every $\alpha >0$. Then we have
\begin{equation} 
N_2^{\tor}(X; H_{12, A}) \leq N_2^{\tor}(X) 
\ll_{\ve} X (\log X)^{1 + \ve}.
\end{equation}
\end{theorem}

\begin{proof}
The inequality \eqref{eq42f} and Corollary \ref{cor43f} imply that
\begin{equation*}
\begin{split}
A_2(X; B) & \leq \frac{X^{\frac{1}{2}}}{B^{\frac{1}{2}}} 
\sum_{\substack{f \in \Sqf_3 \\ f < (2B)^{\frac{1}{2}} }} 2^{w(f)} \tau(f)  g(\frac{2B}{f^2}) \\
& \ll_{\ve} \frac{X^{\frac{1}{2}}}{B^{\frac{1}{2}}} 
\sum_{\substack{f \in \Sqf_3 \\ f < (2B)^{\frac{1}{2}} }} 2^{w(f)} \tau(f) \cdot \frac{2B}{f^2} \log(\frac{2B}{f^2})^{1+\ve} \\
& \ll_{\ve} X^{\frac{1}{2}} (\log 2B)^{1+\ve} B^{\frac{1}{2}} \cdot \sum_{f} \frac{2^{w(f)} \tau(f)}{f^2} \\
& \ll_{\ve} X^{\frac{1}{2}} (\log X)^{1+\ve} B^{\frac{1}{2}}
\end{split}
\end{equation*}
for $B \leq X$. The proof can be completed as in the previous theorem. 
\end{proof}

\section*{Acknowledgments}

The author was supported by the National Research Foundation of Korea (NRF) grant funded by the Korea government (MSIT) (No. RS-2024-00334558 and No. RS-2025-02262988). 
We thank Jordan Ellenberg for pointing out that Conjecture \ref{conj3b} is a consequence of a more general conjecture in \cite{EV05}, and Joachim König for his corrections to this paper. 
We also thank Sungmun Cho, Frank Thorne, Jacob Tsimerman and Melanie Matchett Wood for their helpful comments. 


\end{document}